\documentclass[a4paper,11pt]{article}

\usepackage{geometry}
\usepackage{amsmath,amssymb,amsthm,mathtools}
\usepackage{mathrsfs}
\usepackage{tikz}
\usetikzlibrary{arrows.meta}
\usepackage{subcaption}
\usepackage[colorlinks=true,linkcolor=blue,citecolor=blue,urlcolor=blue]{hyperref}
\usepackage{float}
\numberwithin{equation}{section}

\theoremstyle{plain}
\newtheorem{theorem}{Theorem}[section]
\newtheorem{lemma}[theorem]{Lemma}
\newtheorem{proposition}[theorem]{Proposition}

\theoremstyle{remark}

\theoremstyle{definition}
\newtheorem{definition}[theorem]{Definition}

\begin{document}
	
	\title{Local well-posedness for nonlinear Dirac equation\\ on $N$-star metric graphs}
	
	\author{Huichao Xing$^{1,2}$\quad Zhipeng Yang$^{2,3}$\thanks{Corresponding author: yangzhipeng326@163.com.}\\
		{\small $^{1}$ Faculty of Education, Yunnan Normal University, Kunming, China}\\
		{\small $^{2}$ Yunnan Key Laboratory of Modern Analytical Mathematics and Applications, Kunming, China.}\\
		{\small $^{3}$ Department of Mathematics, Yunnan Normal University, Kunming, China}}
	
	\date{}
	\maketitle
	
	\begin{abstract}
		We consider the Cauchy problem for the nonlinear Dirac equation on a noncompact $N$--star metric graph $G$,
		\[
		\mathrm{i}\partial_t \psi = D\psi - |\psi|^{p-2}\psi,
		\qquad
		\psi(0)=\psi_0,
		\]
		where $p\ge3$, $\psi:\mathbb{R}\times G\to\mathbb{C}^2$ and $D$ denotes the self-adjoint Dirac--Kirchhoff operator on $G$.
		Using Bourgain-type spaces defined through the spectral resolution of $D$, together with elementary $L^\infty$ bounds for the Dirac flow and fractional Nemytskii estimates below the trace threshold, we prove local well-posedness for initial data
		\[
		\psi_0\in H_D^s(G)\cap L^\infty(G;\mathbb C^2),
		\qquad 0\le s<\frac12 .
		\]
		The corresponding solution belongs to
		\[
		C([0,T];H_D^s(G))\cap X_T^{s,b}\cap L^\infty([0,T]\times G).
		\]
		Moreover, $\|\psi(t)\|_{L^2(G;\mathbb{C}^2)}$ is conserved along the solution on the existence interval.
		We also establish a blow-up alternative in the combined $H_D^s$ and space-time $L^\infty$ control norm.
	\end{abstract}

	

	\section{Introduction and Main Results}\label{Sec1}
	
	The Dirac equation is a basic model in relativistic quantum mechanics. It describes spin-$\frac12$ particles through a first-order dispersive dynamics; see \cite{Thaller1992}. In $\mathbb{R}^d$ the free dynamics is typically written as
	\[
	i\partial_t\psi = D_m\psi,
	\qquad
	D_m = -i\alpha\cdot\nabla + m\beta,
	\]
	where $\alpha=(\alpha_1,\dots,\alpha_d)$ and $\beta$ are Dirac matrices.
	Nonlinear Dirac equations arise when self-interactions are introduced at the level of the spinor field, leading to models of the form
	\[
	i\partial_t\psi = D_m\psi - \mathcal{N}(\psi).
	\]
	Classical examples include the Soler model \cite{Soler1970},
	\[
	\mathcal{N}(\psi)=g(\bar\psi\psi)\,\beta\psi,
	\]
	and the Gross--Neveu model \cite{GrossNeveu1974},
	\[
	\mathcal{N}(\psi)=g(\bar\psi\psi)\,\psi.
	\]
	In this paper we focus on a Kerr-type power nonlinearity,
	\[
	\mathcal{N}(\psi)=|\psi|^{p-2}\psi,
	\qquad p\ge3,
	\]
	which is gauge invariant and is a standard model nonlinearity in dispersive equations.
	
	Metric graphs, also called quantum graphs, provide an effective framework for wave propagation on networks, where the dynamics is one-dimensional along edges and the interaction at junctions is encoded by vertex conditions. A basic linear model is generated by the graph Laplacian $-\Delta_G$ with Kirchhoff conditions, which yields the Schr\"odinger flow
	\[
	i\partial_t u=-\Delta_G u,\qquad u(0)=u_0,
	\]
	and the wave equation
	\[
	\partial_t^2 u-\Delta_G u=0.
	\]
	Nonlinear dispersive models on graphs, most notably the nonlinear Schr\"odinger equation
	\[
	i\partial_t u=-\Delta_G u-\lvert u\rvert^{q-2}u,
	\]
	have been studied on both compact and noncompact graphs, where topology and vertex coupling affect the nonlinear dynamics. We refer to \cite{BerkolaikoKuchment2013,Kuchment2008} for background on quantum graphs and to \cite{AdamiCacciapuotiFincoNoja2012,Noja2014} for nonlinear Schr\"odinger dynamics on star graphs and related geometries.
	
	Dirac operators on metric graphs have been investigated from the viewpoint of self-adjoint realizations, together with spectral and scattering theory, see for example \cite{BolteHarrison2003,BullaTrenkler1990}. In the nonlinear setting, Borrelli, Carlone, and Tentarelli \cite{Borrelli2021} introduced a self-adjoint Dirac operator on a noncompact metric graph $G$ with Kirchhoff-type vertex conditions and studied the nonlinear Dirac dynamics
	\[
	i\partial_t\psi = D_G\psi-|\psi|^{p-2}\psi,\qquad p>2,
	\]
	together with the localized variant
	\[
	i\partial_t\psi = D_G\psi-\chi_K|\psi|^{p-2}\psi,\qquad p>2,
	\]
	where $K$ denotes the compact core of $G$ and $\chi_K$ is its characteristic function. For initial data $\psi_0\in \mathrm{dom}(D_G)$ they proved a maximal well-posedness theory, providing a unique solution
	\[
	\psi\in C\!\bigl(I,\mathrm{dom}(D_G)\bigr)\cap C^1\!\bigl(I,L^2(G;\mathbb{C}^2)\bigr)
	\]
	on a maximal interval $I$, together with conservation of charge and energy and a blow-up alternative in the graph norm, see \cite[Theorem~1.1]{Borrelli2021}.
	
	Here and below, once the graph has been fixed, we write $D$ for this Dirac--Kirchhoff operator; this is the same operator denoted by $D_G$ in \cite{Borrelli2021}. Related recent variational results for nonlinear Dirac equations on metric or quantum graphs include \cite{HeJi2025Normalized,HeJi2025Limit,GuLiRuzhanskyYang2025}.
	
	The present work concerns the Cauchy problem for nonlinear Dirac equations on metric graphs at Sobolev regularities below the trace threshold. On the line, \(X^{s,b}\)-type methods adapted to the characteristic structure of one-dimensional Dirac operators were developed by Machihara \cite{Machihara2005} and Machihara--Nakanishi--Tsugawa \cite{MachiharaNakanishiTsugawa2010}. On an \(N\)-star graph, the vertex condition couples the half-line components, and the standard space--time Fourier analysis on \(\mathbb R\times\mathbb R\) is not directly available. We work instead with spectral Bourgain spaces associated with the self-adjoint Dirac--Kirchhoff operator \(D\). For the Kerr-type nonlinearity considered here, the nonlinear estimates are derived from fractional Nemytskii estimates in \(H^s\cap L^\infty\) and from the equivalence between \(H_D^s(G)\) and the edgewise Sobolev space for \(0\le s<1/2\).
	
	Low-regularity dispersive equations on star graphs have also been studied by Fourier-restriction or boundary-forcing methods for other models; see, for instance, Cavalcante's work on the Korteweg--de Vries equation on a metric star graph \cite{Cavalcante2018}, the quarter-plane and star-graph analysis of Capistrano-Filho--Cavalcante--Gallego for biharmonic Schr\"odinger equations \cite{CapistranoFilhoCavalcanteGallego2020}, and their fourth-order Schr\"odinger work on star graphs \cite{CapistranoFilhoCavalcanteGallego2022}. Here the Bourgain-type spaces are defined through the spectral resolution of the self-adjoint Dirac--Kirchhoff operator $D$, rather than through boundary forcing operators.
	
	The proof uses three ingredients. First, the Dirac--Kirchhoff operator is diagonalized in the finite-dimensional edge index into half-line Dirac channels, and the corresponding boundary-adapted reflections reduce the linear estimates to the line. Second, for $0\le s<1/2$ the operator Sobolev space $H_D^s(G)$ is equivalent to the edgewise Sobolev space, so the vertex trace is not part of the function-space structure at this regularity. Third, the boundedness of the initial data permits a contraction argument in $X_T^{s,b}\cap L^\infty([0,T]\times G)$ through a Nemytskii estimate for $z\mapsto |z|^{p-2}z$. The mixed-sign quadratic model studied in \cite{XingYang2026MixedSign} is based on a different bilinear structure; the present paper treats the Kerr-type power nonlinearity by composition estimates in $H^s\cap L^\infty$.
	
	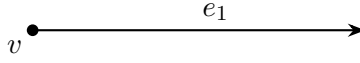
\begin{figure}[H]
		\centering
		\begin{tikzpicture}[scale=1.1, line cap=round, line join=round, >=Stealth]
			\fill (0,0) circle (2pt);
			\node[below left] at (0,0) {$v$};
			\draw[thick,->] (0,0) -- (4,0);
			\node[above] at (2.2,0) {$e_1$};
		\end{tikzpicture}
		\caption{The $1$-star metric graph $G$}
		\label{Fig1.1}
	\end{figure}
	
	\begin{figure}[H]
		\centering
		\begin{tikzpicture}[scale=1.1, line cap=round, line join=round, >=Stealth]
			\fill (0,0) circle (2pt);
			\node[below] at (0,0) {$v$};
			\draw[thick,->] (0,0) -- (4,0);
			\draw[thick,->] (0,0) -- (-4,0);
			\node[above] at (2.2,0) {$e_1$};
			\node[above] at (-2.2,0) {$e_2$};
		\end{tikzpicture}
		\caption{The $2$-star metric graph $G$}
		\label{Fig1.2}
	\end{figure}
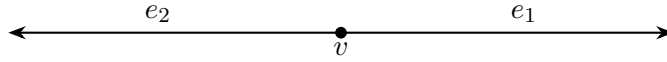
	
	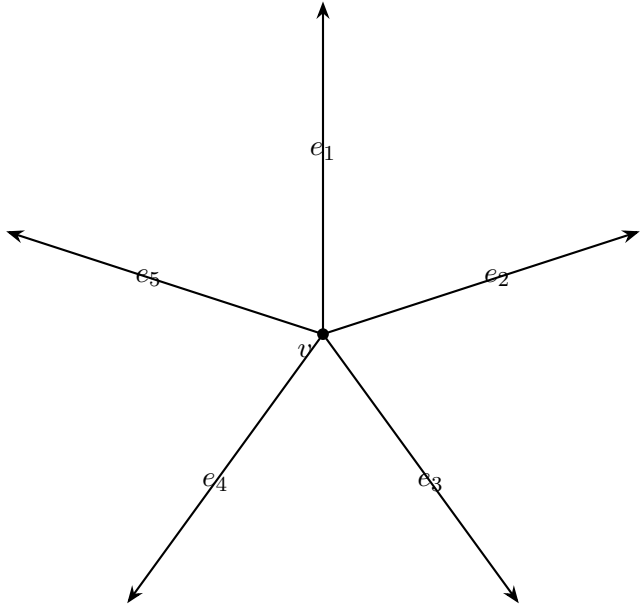
\begin{figure}[H]
		\centering
		\begin{tikzpicture}[scale=1.1, line cap=round, line join=round, >=Stealth]
			\fill (0,0) circle (2pt);
			\node[below left] at (0,0) {$v$};
			
			\foreach \ang/\lab in {90/{$e_1$},18/{$e_2$},-54/{$e_3$},-126/{$e_4$},162/{$e_5$}} {
				\draw[thick,->] (0,0) -- (\ang:4);
				\node at (\ang:2.2) {\lab};
			}
		\end{tikzpicture}
		\caption{The $5$-star metric graph $G$}
		\label{Fig1.3}
	\end{figure}
	
	Let $G$ be the noncompact $N$-star metric graph obtained by gluing $N$ copies of $\mathbb{R}_+$ at a common vertex; see Figures~\ref{Fig1.1}--\ref{Fig1.3}.
	Let $m\ge 0$ and let $D$ be the Dirac--Kirchhoff operator acting edgewise by
	\[
	(D\psi)_e=\left(-i\sigma_1\partial_x+m\sigma_3\right)\psi_e
	\quad\text{on }\mathbb{R}_+^{(e)},
	\]
	with vertex condition \eqref{eq2.1} and domain \eqref{eq2.2}. For $s\in\mathbb{R}$ we work in the operator Sobolev space $H_D^s(G)$ endowed with its spectral norm. We study the Cauchy problem
	\begin{equation}\label{eq1.1}
		\begin{cases}
			i\partial_t\psi = D\psi - |\psi|^{p-2}\psi, & (t,x)\in\mathbb{R}\times G,\\
			\psi(0,x)=\psi_0(x), & x\in G,
		\end{cases}
	\end{equation}
	where $p\ge3$.

	\begin{theorem}\label{Thm1.1}
		Let $p\ge3$ and let $0\le s<\frac12$. Let
		\[
		\psi_0\in H_D^s(G)\cap L^\infty(G;\mathbb C^2).
		\]
		There exist $T=T(\|\psi_0\|_{H_D^s(G)}+\|\psi_0\|_{L^\infty(G)};G)>0$ and $b\in\left(\frac12,1\right)$ such that \eqref{eq1.1} admits a unique solution
		\[
		\psi\in C([0,T];H_D^s(G))\cap X_T^{s,b}\cap L^\infty([0,T]\times G).
		\]
		The data-to-solution map is locally Lipschitz from $H_D^s(G)\cap L^\infty(G)$ to
		\[
		C([0,T];H_D^s(G))\cap X_T^{s,b}\cap L^\infty([0,T]\times G).
		\]
		Moreover,
		\[
		\|\psi(t)\|_{L^2(G;\mathbb{C}^2)}=\|\psi_0\|_{L^2(G;\mathbb{C}^2)}
		\quad\text{for all }t\in[0,T].
		\]
		Let $[0,T^\ast)$ be the maximal forward lifespan of $\psi$ in the class
		\[
		C([0,T);H_D^s(G))\cap X_{\mathrm{loc}}^{s,b}([0,T))\cap L^\infty_{\mathrm{loc}}([0,T)\times G).
		\]
		If $T^\ast<\infty$ and
		\[
		\sup_{t<T^\ast}\|\psi(t)\|_{H_D^s(G)}+
		\|\psi\|_{L^\infty([0,T^\ast)\times G)}<\infty,
		\]
		then there exists $\delta>0$ such that $\psi$ extends to a solution on $[0,T^\ast+\delta)$.
	\end{theorem}
	
	Here \(X_{\mathrm{loc}}^{s,b}([0,T))\) means that the restriction of the
	solution to every compact subinterval \([0,T_0]\subset [0,T)\) belongs to
	\(X_{T_0}^{s,b}\). The notation \(L^\infty_{\mathrm{loc}}\) is understood in
	the same way.
	
	The restriction \(0\le s<1/2\) places the analysis below the trace threshold.
	In this range \(H_D^s(G)\) is equivalent to the edgewise Sobolev space
	\[
	\bigoplus_{e=1}^N H^s(\mathbb R_+;\mathbb C^2).
	\]
	Thus the required composition estimates can be carried out edge by edge in
	\(H^s\cap L^\infty\), and no compatibility condition for the nonlinear term
	at the vertex enters the argument. The \(L^\infty\) control is used in the
	Nemytskii estimates for the power-type nonlinearity.
	
	The paper is organized as follows. We first recall the spectral setting of the Dirac--Kirchhoff operator and introduce the Bourgain spaces $X^{s,b}$ and their restriction versions $X_T^{s,b}$. We then establish the linear cutoff and Duhamel estimates. Next we prove the edge-mode reduction, the $L^\infty$ linear estimate, and the $H^s\cap L^\infty$ nonlinear estimates used in the contraction argument. Finally, we complete the proof of Theorem~\ref{Thm1.1}, including charge conservation and the blow-up alternative.
	
	\section{Preliminaries and Function Spaces}\label{Sec2}
	
	Let $G$ be the noncompact $N$--star metric graph obtained by gluing $N$ copies of
	$\mathbb{R}_+$ at a common vertex $v$; see Figures~\ref{Fig1.1}--\ref{Fig1.3} for typical examples.
	We write
	\[
	G=\bigcup_{e=1}^N \mathbb{R}_+^{(e)},\qquad x\in(0,\infty)\ \text{on each edge}.
	\]
	A spinor $\psi=(\varphi,\chi)^{\mathsf T}$ on $G$ is given edgewise by
	$\psi_e=(\varphi_e,\chi_e)^{\mathsf T}:\mathbb{R}_+^{(e)}\to\mathbb{C}^2$.
	Set
	\[
	L^2(G;\mathbb{C}^2)=\bigoplus_{e=1}^N L^2(\mathbb{R}_+^{(e)};\mathbb{C}^2),
	\qquad
	\|\psi\|_{L^2(G)}^2=\sum_{e=1}^N\int_0^\infty |\psi_e(x)|^2\,dx,
	\]
	where $|\cdot|$ is the Euclidean norm on $\mathbb{C}^2$, and
	$\langle u,v\rangle_{\mathbb{C}^2}=u^*v$ denotes the standard Hermitian inner product,
	linear in the second argument.
	
	For functions on an edge we use the notation $C_c^\infty([0,\infty))$ for smooth functions
	compactly supported in $[0,\infty)$.
	We use the Pauli matrices
	\[
	\sigma_1=
	\begin{pmatrix}
		0 & 1\\
		1 & 0
	\end{pmatrix},
	\qquad
	\sigma_3=
	\begin{pmatrix}
		1 & 0\\
		0 & -1
	\end{pmatrix}.
	\]
	Fix $m\ge 0$. On each edge we consider the free Dirac expression
	\[
	D_e=-i\sigma_1\partial_x+m\sigma_3.
	\]
	
	At the vertex we impose the Dirac--Kirchhoff conditions
	\begin{equation}\label{eq2.1}
		\varphi_1(0)=\cdots=\varphi_N(0),
		\qquad
		\sum_{e=1}^N \chi_e(0)=0.
	\end{equation}
	
	We set the first-order Sobolev space on the graph
	\[
	H^1(G;\mathbb{C}^2)=\bigoplus_{e=1}^N H^1(\mathbb{R}_+^{(e)};\mathbb{C}^2),
	\]
	so that the traces $\psi_e(0)$ are well-defined for $\psi\in H^1(G;\mathbb{C}^2)$.
	We define the maximal Dirac operator $D_{\max}$ on $L^2(G;\mathbb{C}^2)$ by
	\[
	(D_{\max}\psi)_e = D_e\psi_e,
	\qquad
	\mathrm{Dom}(D_{\max}) = H^1(G;\mathbb{C}^2).
	\]
	We define the Dirac--Kirchhoff realization $D$ as the restriction of $D_{\max}$:
	\begin{equation}\label{eq2.2}
		(D\psi)_e = D_e\psi_e,
		\qquad
		\mathrm{Dom}(D)=\Bigl\{\psi\in H^1(G;\mathbb{C}^2)\,:\,\psi \text{ satisfies }\eqref{eq2.1}\Bigr\}.
	\end{equation}

	The integration-by-parts identity and the self-adjointness of the Dirac--Kirchhoff realization are standard; see, for example, \cite{BolteHarrison2003,BullaTrenkler1990,Borrelli2021}. We include the details only to fix notation and the convention for the boundary form used later.
	
	\begin{lemma}\label{Lem2.1}
		For all $\psi,\phi\in \mathrm{Dom}(D_{\max})=H^1(G;\mathbb{C}^2)$ one has
		\begin{equation}\label{eq2.3}
			\langle D_{\max}\psi,\phi\rangle_{L^2(G)}-\langle \psi,D_{\max}\phi\rangle_{L^2(G)}
			=-i\sum_{e=1}^N \bigl\langle \psi_e(0),\sigma_1\phi_e(0)\bigr\rangle_{\mathbb{C}^2}.
		\end{equation}
		In particular, the operator $D$ in \eqref{eq2.2} is symmetric on $L^2(G;\mathbb{C}^2)$.
	\end{lemma}

	\begin{proof}
		We first prove \eqref{eq2.3} for edgewise test functions by integration by parts on each edge.
		Let $\psi,\phi$ be such that for every $e$ one has
		$\psi_e,\phi_e\in C_c^\infty([0,\infty);\mathbb{C}^2)$. On each edge $\mathbb{R}_+^{(e)}$ we have
		\[
		\int_0^\infty \left\langle -i\sigma_1\partial_x\psi_e(x),\phi_e(x)\right\rangle_{\mathbb{C}^2}\,dx
		-\int_0^\infty \left\langle \psi_e(x),-i\sigma_1\partial_x\phi_e(x)\right\rangle_{\mathbb{C}^2}\,dx
		=-i\left\langle \psi_e(0),\sigma_1\phi_e(0)\right\rangle_{\mathbb{C}^2},
		\]
		where the boundary term at $+\infty$ vanishes since $\psi_e,\phi_e$ are compactly supported.
		Moreover, since $\sigma_3$ is Hermitian,
		\[
		\int_0^\infty \langle m\sigma_3\psi_e,\phi_e\rangle_{\mathbb{C}^2}\,dx
		-\int_0^\infty \langle \psi_e,m\sigma_3\phi_e\rangle_{\mathbb{C}^2}\,dx=0.
		\]
		Summing over $e=1,\dots,N$ yields \eqref{eq2.3} for such test functions.
		
		Now let $\psi,\phi\in H^1(G;\mathbb{C}^2)$. By the standard density of $C_c^\infty([0,\infty))$ in
		$H^1(\mathbb{R}_+)$ on each edge, we can choose sequences $\psi^{(n)},\phi^{(n)}\in H^1(G;\mathbb{C}^2)$
		such that for each $e$,
		$\psi_e^{(n)},\phi_e^{(n)}\in C_c^\infty([0,\infty);\mathbb{C}^2)$ and
		\[
		\psi^{(n)}\to\psi,\qquad \phi^{(n)}\to\phi
		\quad\text{in }H^1(G;\mathbb{C}^2).
		\]
		Then $D_{\max}\psi^{(n)}\to D_{\max}\psi$ and $D_{\max}\phi^{(n)}\to D_{\max}\phi$ in $L^2(G;\mathbb{C}^2)$.
		The trace map $H^1(\mathbb{R}_+;\mathbb{C}^2)\to \mathbb{C}^2$, $u\mapsto u(0)$, is continuous on each edge, hence
		$\psi^{(n)}_e(0)\to\psi_e(0)$ and $\phi^{(n)}_e(0)\to\phi_e(0)$ in $\mathbb{C}^2$ for every $e$.
		Passing to the limit in \eqref{eq2.3} for $(\psi^{(n)},\phi^{(n)})$ gives \eqref{eq2.3} for all
		$\psi,\phi\in H^1(G;\mathbb{C}^2)$.
		
		Finally, if $\psi,\phi\in\mathrm{Dom}(D)$, then the vertex conditions \eqref{eq2.1} imply
		\[
		\sum_{e=1}^N \left\langle \psi_e(0),\sigma_1\phi_e(0)\right\rangle_{\mathbb{C}^2}=0,
		\]
		hence the boundary term in \eqref{eq2.3} vanishes and $D$ is symmetric.
	\end{proof}

	\begin{proposition}\label{Prop2.2}
		The Dirac--Kirchhoff operator $D$ defined by \eqref{eq2.2} is self-adjoint on $L^2(G;\mathbb{C}^2)$.
	\end{proposition}
	
	\begin{proof}
		By Lemma~\ref{Lem2.1}, $D$ is symmetric, hence $D\subset D^*$.
		
		Let $\psi\in\mathrm{Dom}(D^*)$. Then there exists $f\in L^2(G;\mathbb{C}^2)$ such that
		\[
		\langle D\phi,\psi\rangle_{L^2(G)}=\langle \phi,f\rangle_{L^2(G)}
		\quad\text{for all }\phi\in\mathrm{Dom}(D).
		\]
		Fix an edge $e$ and take $\phi\in\mathrm{Dom}(D)$ supported in the interior of that edge, with
		$\phi_e\in C_c^\infty((0,\infty);\mathbb{C}^2)$ and $\phi_{e'}=0$ for $e'\ne e$.
		Then $\phi_e(0)=0$, so $\phi$ satisfies \eqref{eq2.1}. Using the above identity and the formal expression
		$D_e=-i\sigma_1\partial_x+m\sigma_3$, we obtain
		\[
		\int_0^\infty \langle D_e\phi_e(x),\psi_e(x)\rangle_{\mathbb{C}^2}\,dx
		=\int_0^\infty \langle \phi_e(x),f_e(x)\rangle_{\mathbb{C}^2}\,dx
		\quad\text{for all }\phi_e\in C_c^\infty((0,\infty);\mathbb{C}^2).
		\]
		This implies $\psi_e\in H^1(\mathbb{R}_+^{(e)};\mathbb{C}^2)$ and $f_e=D_e\psi_e$ in
		$L^2(\mathbb{R}_+^{(e)};\mathbb{C}^2)$. Since $e$ is arbitrary, we obtain
		$\psi\in H^1(G;\mathbb{C}^2)=\mathrm{Dom}(D_{\max})$ and $D^*\psi=D_{\max}\psi$.
		
		For every $\phi\in \mathrm{Dom}(D)$, Lemma~\ref{Lem2.1} gives
		\[
		0=\langle D^*\psi,\phi\rangle_{L^2(G)}-\langle \psi,D\phi\rangle_{L^2(G)}
		=-i\sum_{e=1}^N \bigl\langle \psi_e(0),\sigma_1\phi_e(0)\bigr\rangle_{\mathbb{C}^2}.
		\]
		Write $\psi_e(0)=\binom{a_e}{b_e}$ and $\phi_e(0)=\binom{\alpha}{\beta_e}$, where $\alpha\in\mathbb{C}$ and
		$(\beta_e)_{e=1}^N\in\mathbb{C}^N$ satisfy $\sum_{e=1}^N\beta_e=0$ because $\phi\in\mathrm{Dom}(D)$.
		Indeed, fix $\eta\in C_c^\infty([0,\infty))$ with $\eta(0)=1$ and set
		$\phi_e(x)=(\alpha,\beta_e)^{\mathsf T}\eta(x)$; then $\phi\in\mathrm{Dom}(D)$ and $\phi_e(0)=(\alpha,\beta_e)^{\mathsf T}$.
		
		Choosing such $\phi$ with prescribed $(\alpha,\beta_e)$, we obtain
		\[
		\sum_{e=1}^N \bigl(\overline{a_e}\,\beta_e+\overline{b_e}\,\alpha\bigr)=0
		\quad\text{for all }\alpha\in\mathbb{C},\ (\beta_e)\in\mathbb{C}^N \text{ with } \sum_{e=1}^N\beta_e=0.
		\]
		Taking $\alpha=0$ yields $\sum_{e=1}^N \overline{a_e}\,\beta_e=0$ for all $(\beta_e)$ with zero sum, hence
		$a_1=\cdots=a_N$. Taking $\beta_e=0$ and arbitrary $\alpha$ yields $\sum_{e=1}^N b_e=0$.
		Therefore $\psi$ satisfies \eqref{eq2.1}, so $\psi\in\mathrm{Dom}(D)$, and
		$\mathrm{Dom}(D^*)\subset \mathrm{Dom}(D)$.
		
		Hence $\mathrm{Dom}(D^*)=\mathrm{Dom}(D)$ and $D$ is self-adjoint.
	\end{proof}

	\subsection{Spectral transform and Bourgain spaces}
	
	For $s\in\mathbb{R}$ we define the operator Sobolev space $H_D^s(G)$ as follows.
	If $s\ge 0$, set
	\[
	H_D^s(G)=\mathrm{Dom}\bigl(\langle D\rangle^s\bigr),
	\qquad
	\|\psi\|_{H_D^s(G)}=\|\langle D\rangle^s\psi\|_{L^2(G)},
	\qquad
	\langle D\rangle=(I+D^2)^{1/2}.
	\]
	If $s<0$, define $H_D^s(G)$ as the completion of $L^2(G;\mathbb{C}^2)$ with respect to the norm
	$\|\psi\|_{H_D^s(G)}=\|\langle D\rangle^s\psi\|_{L^2(G)}$.
	
	We denote by $E_D(\cdot)$ the projection-valued spectral measure of the self-adjoint operator $D$. The full spectral measure, including a possible atom at $\lambda=0$, is used throughout. Since the Bourgain weights below are finite at $\lambda=0$, zero spectral components are covered by the same framework.
	
	We write
	\[
	\langle \xi\rangle=(1+|\xi|^{2})^{1/2}.
	\]
	By the spectral theorem for the self-adjoint operator $D$, there exist a Borel measure $\mu_D$ on $\mathbb{R}$,
	a measurable family of Hilbert spaces $\{\mathcal{H}_\lambda\}_{\lambda\in\mathbb{R}}$,
	and a unitary map
	\[
	\mathcal{F}_D:L^2(G;\mathbb{C}^2)\to L^2(\mathbb{R},d\mu_D;\mathcal{H}_\lambda)
	\]
	such that the following holds.
	For any Borel measurable function $m:\mathbb{R}\to\mathbb{C}$, the operator $m(D)$ is defined by functional calculus with domain
	\[
	\mathrm{Dom}(m(D))
	=\Bigl\{f\in L^2(G;\mathbb{C}^2):\ m(\lambda)(\mathcal{F}_D f)(\lambda)\in L^2(\mathbb{R},d\mu_D;\mathcal{H}_\lambda)\Bigr\},
	\]
	and for every $f\in\mathrm{Dom}(m(D))$ one has
	\[
	\bigl(\mathcal{F}_D(m(D)f)\bigr)(\lambda)=m(\lambda)\,(\mathcal{F}_D f)(\lambda)
	\quad\text{for }\mu_D\text{-a.e. }\lambda\in\mathbb{R}.
	\]
	In particular, for every $f\in \mathrm{Dom}(D)$,
	\[
	(\mathcal{F}_D(D f))(\lambda)=\lambda(\mathcal{F}_D f)(\lambda)
	\quad\text{for }\mu_D\text{-a.e. }\lambda\in\mathbb{R}.
	\]
	More generally, for every $s\in\mathbb{R}$ and every $f\in \mathrm{Dom}(\langle D\rangle^{s})$ one has
	\[
	\bigl(\mathcal{F}_D(\langle D\rangle^{s} f)\bigr)(\lambda)
	=\langle\lambda\rangle^{s}(\mathcal{F}_D f)(\lambda)
	\quad\text{for }\mu_D\text{-a.e. }\lambda\in\mathbb{R}.
	\]
	By the unitarity of $\mathcal{F}_D$, this yields
	\begin{equation}\label{eq2.4}
		\|f\|_{H_D^s(G)}^2
		=\|\langle D\rangle^{s} f\|_{L^2(G)}^{2}
		=\int_{\mathbb{R}}\langle\lambda\rangle^{2s}\|(\mathcal{F}_D f)(\lambda)\|_{\mathcal{H}_\lambda}^2\,d\mu_D(\lambda),
	\end{equation}
	for all $f\in \mathrm{Dom}(\langle D\rangle^{s})$. By density, \eqref{eq2.4} extends to all $f\in H_D^s(G)$.
	
	For a spacetime function $u:\mathbb{R}\times G\to\mathbb{C}^2$ we define its joint time--spectral transform by
	\[
	\widetilde u(\tau,\lambda)=(\mathcal{F}_t\mathcal{F}_D u)(\tau,\lambda),
	\qquad
	\mathcal{F}_t u(\tau)=\frac{1}{\sqrt{2\pi}}\int_{\mathbb{R}}e^{-it\tau}u(t)\,dt,
	\]
	where $\mathcal{F}_D$ acts in the space variable for each fixed $t$, and $\mathcal{F}_t$ is understood as the Fourier transform
	in $t$ of Hilbert space-valued functions, in the usual Bochner sense.
	
	We fix the dense core
	\[
	\mathcal{U}
	=\bigcup_{M\in\mathbb{N}} C_c^\infty\bigl(\mathbb{R};\mathrm{Dom}(\langle D\rangle^M)\bigr),
	\]
	on which all the transforms and multipliers below are well-defined.
	
	\begin{definition}\label{Def2.3}
		Let $s,b\in\mathbb{R}$.
		We define $X^{s,b}(G)$ as the completion of $\mathcal{U}$ under the norm
		\begin{equation}\label{eq2.5}
			\|u\|_{X^{s,b}}^2
			=\int_{\mathbb{R}}\int_{\mathbb{R}}
			\langle\lambda\rangle^{2s}\langle\tau+\lambda\rangle^{2b}
			\|\widetilde u(\tau,\lambda)\|_{\mathcal{H}_\lambda}^2
			\,d\tau\,d\mu_D(\lambda).
		\end{equation}
	\end{definition}

	\begin{lemma}\label{Lem2.4}
		Let $u\in\mathcal{U}$ and set $v(t)=e^{itD}u(t)$. Then $v\in\mathcal{U}$ and
		\[
		\|u\|_{X^{s,b}}^{2}
		=\int_{\mathbb{R}}\int_{\mathbb{R}}
		\langle\lambda\rangle^{2s}\langle\tau\rangle^{2b}
		\|(\mathcal{F}_t\mathcal{F}_D v)(\tau,\lambda)\|_{\mathcal{H}_\lambda}^{2}
		\,d\tau\,d\mu_D(\lambda).
		\]
	\end{lemma}
	
	\begin{proof}
		Let $u\in\mathcal{U}$ and define $v(t)=e^{itD}u(t)$.
		Since $e^{itD}$ commutes with $\langle D\rangle^M$ by functional calculus and is unitary on $L^2(G;\mathbb{C}^2)$,
		one has $u(t)\in\mathrm{Dom}(\langle D\rangle^M)$ for all $t$ if and only if $v(t)\in\mathrm{Dom}(\langle D\rangle^M)$ for all $t$.
		Hence $v\in\mathcal{U}$.
		
		Using the spectral representation of $D$, for $\mu_D$-a.e. $\lambda$ we have, for each $t\in\mathbb{R}$,
		\[
		(\mathcal{F}_D v)(t,\lambda)=e^{it\lambda}(\mathcal{F}_D u)(t,\lambda).
		\]
		Taking $\mathcal{F}_t$ yields
		\[
		(\mathcal{F}_t\mathcal{F}_D v)(\tau,\lambda)
		=(\mathcal{F}_t\mathcal{F}_D u)(\tau-\lambda,\lambda)
		=\widetilde u(\tau-\lambda,\lambda).
		\]
		Therefore
		\[
		\int_{\mathbb{R}}\langle\tau\rangle^{2b}
		\|(\mathcal{F}_t\mathcal{F}_D v)(\tau,\lambda)\|_{\mathcal{H}_\lambda}^{2}\,d\tau
		=
		\int_{\mathbb{R}}\langle\tau\rangle^{2b}
		\|\widetilde u(\tau-\lambda,\lambda)\|_{\mathcal{H}_\lambda}^{2}\,d\tau.
		\]
		Substituting $\sigma=\tau-\lambda$ (so that $\tau=\sigma+\lambda$) gives
		\[
		\int_{\mathbb{R}}\langle\tau\rangle^{2b}
		\|\widetilde u(\tau-\lambda,\lambda)\|_{\mathcal{H}_\lambda}^{2}\,d\tau
		=
		\int_{\mathbb{R}}\langle\sigma+\lambda\rangle^{2b}
		\|\widetilde u(\sigma,\lambda)\|_{\mathcal{H}_\lambda}^{2}\,d\sigma.
		\]
		Multiplying by $\langle\lambda\rangle^{2s}$, integrating in $\lambda$ against $d\mu_D(\lambda)$,
		and using \eqref{eq2.5} yields the claimed identity.
	\end{proof}
	
	\begin{definition}\label{Def2.5}
		Let $T>0$.
		The restriction space $X_T^{s,b}$ consists of all functions $u$ defined on $[0,T]$ such that
		\[
		\|u\|_{X_T^{s,b}}
		=\inf\{\|U\|_{X^{s,b}}:\ U\in X^{s,b},\ U(t)=u(t)\ \text{for }t\in[0,T]\}
		<\infty.
		\]
	\end{definition}
	
	\begin{lemma}\label{Lem2.6}
		For every $s,b\in\mathbb{R}$ and $T>0$, the space $X_T^{s,b}$ is a Banach space.
		Moreover, if $0<T_1\le T_2$ and $u$ is defined on $[0,T_2]$, then
		\[
		\|u|_{[0,T_1]}\|_{X_{T_1}^{s,b}}\le \|u\|_{X_{T_2}^{s,b}}.
		\]
	\end{lemma}
	
	\begin{proof}
		Let $X=X^{s,b}(G)$. Define the linear map
		\[
		r_T:X\to X_T^{s,b},\qquad r_T(U)=U|_{[0,T]}.
		\]
		By Definition~\ref{Def2.5}, for every $U\in X$ the function $r_T(U)$ admits $U$ itself as an extension, hence
		\[
		\|r_T(U)\|_{X_T^{s,b}}\le \|U\|_{X}.
		\]
		Therefore $r_T$ is bounded. Set
		\[
		\mathcal{N}_T=\ker r_T=\{W\in X:\ W|_{[0,T]}=0\}.
		\]
		Since $r_T$ is continuous, $\mathcal{N}_T$ is a closed subspace of $X$.
		
		The map $r_T$ is surjective by the definition of $X_T^{s,b}$.
		Moreover, two elements $U,V\in X$ have the same restriction on $[0,T]$ if and only if $U-V\in\mathcal{N}_T$.
		Hence $X_T^{s,b}$ is canonically isometric to the quotient space $X/\mathcal{N}_T$ endowed with the quotient norm.
		Since $X$ is Banach and $\mathcal{N}_T$ is closed, the quotient $X/\mathcal{N}_T$ is Banach, and therefore $X_T^{s,b}$ is Banach.
		
		Finally, let $0<T_1\le T_2$ and let $u$ be defined on $[0,T_2]$.
		For every extension $U\in X$ with $U=u$ on $[0,T_2]$, the same $U$ satisfies $U=u|_{[0,T_1]}$ on $[0,T_1]$.
		Taking the infimum over all such $U$ yields
		\[
		\|u|_{[0,T_1]}\|_{X_{T_1}^{s,b}}\le \|u\|_{X_{T_2}^{s,b}}.
		\]
	\end{proof}

	\begin{lemma}\label{Lem2.7}
		If $b>\frac12$, then
		\[
		X_T^{s,b}\hookrightarrow C([0,T];H_D^s(G)).
		\]
		More precisely, there exists a constant $C_b$ such that for every $u\in X_T^{s,b}$,
		\begin{equation}\label{eq2.6}
			\sup_{t\in[0,T]}\|u(t)\|_{H_D^s(G)}\le C_b\|u\|_{X_T^{s,b}}.
		\end{equation}
	\end{lemma}
	
	\begin{proof}
		We first prove the global embedding
		\[
		X^{s,b}(G)\hookrightarrow C(\mathbb{R};H_D^s(G))
		\quad\text{for } b>\frac12,
		\]
		together with the estimate
		\begin{equation}\label{eq2.7}
			\sup_{t\in\mathbb{R}}\|U(t)\|_{H_D^s(G)}\le C_b\|U\|_{X^{s,b}},
			\qquad U\in X^{s,b}(G).
		\end{equation}
		
		Step 1. Estimate and continuity for $U\in\mathcal{U}$.
		Let $U\in\mathcal{U}$. For each $\lambda$ we have the Fourier inversion formula in time,
		\[
		(\mathcal{F}_D U)(t,\lambda)
		=\frac1{\sqrt{2\pi}}\int_{\mathbb{R}}e^{it\tau}\widetilde U(\tau,\lambda)\,d\tau
		\quad\text{in }\mathcal{H}_\lambda.
		\]
		By Cauchy--Schwarz in $\tau$ with weight $\langle\tau+\lambda\rangle^{b}$,
		\[
		\|(\mathcal{F}_D U)(t,\lambda)\|_{\mathcal{H}_\lambda}
		\le \frac1{\sqrt{2\pi}}
		\left(\int_{\mathbb{R}}\langle\tau+\lambda\rangle^{-2b}\,d\tau\right)^{1/2}
		\left(\int_{\mathbb{R}}\langle\tau+\lambda\rangle^{2b}\|\widetilde U(\tau,\lambda)\|_{\mathcal{H}_\lambda}^2\,d\tau\right)^{1/2}.
		\]
		Since $b>\frac12$, the constant
		\[
		K_b=\int_{\mathbb{R}}\langle\sigma\rangle^{-2b}\,d\sigma<\infty,
		\qquad
		\int_{\mathbb{R}}\langle\tau+\lambda\rangle^{-2b}\,d\tau=K_b,
		\]
		hence by \eqref{eq2.4},
		\[
		\|U(t)\|_{H_D^s(G)}^2
		=\int_{\mathbb{R}}\langle\lambda\rangle^{2s}\|(\mathcal{F}_D U)(t,\lambda)\|_{\mathcal{H}_\lambda}^2\,d\mu_D(\lambda)
		\le \frac{K_b}{2\pi}\|U\|_{X^{s,b}}^2,
		\]
		which gives \eqref{eq2.7} for $U\in\mathcal{U}$ with $C_b=\sqrt{\frac{K_b}{2\pi}}$.
		
		To prove continuity for $U\in\mathcal{U}$, let $t_n\to t$. For $\mu_D$-a.e.\ $\lambda$,
		\[
		(\mathcal{F}_D U)(t_n,\lambda)-(\mathcal{F}_D U)(t,\lambda)
		=\frac1{\sqrt{2\pi}}\int_{\mathbb{R}}\bigl(e^{it_n\tau}-e^{it\tau}\bigr)\widetilde U(\tau,\lambda)\,d\tau.
		\]
		By Cauchy--Schwarz with the same weight,
		\[
		\|(\mathcal{F}_D U)(t_n,\lambda)-(\mathcal{F}_D U)(t,\lambda)\|_{\mathcal{H}_\lambda}
		\le \frac1{\sqrt{2\pi}}A_n(\lambda)\,B(\lambda),
		\]
		where
		\[
		A_n(\lambda)
		=\left(\int_{\mathbb{R}}|e^{it_n\tau}-e^{it\tau}|^2\langle\tau+\lambda\rangle^{-2b}\,d\tau\right)^{1/2},
		\qquad
		B(\lambda)
		=\left(\int_{\mathbb{R}}\langle\tau+\lambda\rangle^{2b}\|\widetilde U(\tau,\lambda)\|_{\mathcal{H}_\lambda}^2\,d\tau\right)^{1/2}.
		\]
		Since $|e^{it_n\tau}-e^{it\tau}|^2\to 0$ pointwise in $\tau$ and is bounded by $4$, dominated convergence yields
		$A_n(\lambda)\to 0$ for each fixed $\lambda$, and $A_n(\lambda)^2\le 4K_b$ for all $n,\lambda$.
		Therefore,
		\[
		\begin{aligned}
			\|U(t_n)-U(t)\|_{H_D^s(G)}^2
			&=\int_{\mathbb{R}}\langle\lambda\rangle^{2s}
			\|(\mathcal{F}_D U)(t_n,\lambda)-(\mathcal{F}_D U)(t,\lambda)\|_{\mathcal{H}_\lambda}^2\,d\mu_D(\lambda)\\
			&\le \frac1{2\pi}\int_{\mathbb{R}}\langle\lambda\rangle^{2s}A_n(\lambda)^2B(\lambda)^2\,d\mu_D(\lambda).    
		\end{aligned}
		\]
		Since $U\in X^{s,b}$, one has
		\[
		\int_{\mathbb{R}}\langle\lambda\rangle^{2s}B(\lambda)^2\,d\mu_D(\lambda)=\|U\|_{X^{s,b}}^2<\infty.
		\]
		Hence $\langle\lambda\rangle^{2s}A_n(\lambda)^2B(\lambda)^2$ is dominated by $4K_b\,\langle\lambda\rangle^{2s}B(\lambda)^2\in L^1(d\mu_D)$,
		and dominated convergence in $\lambda$ implies $\|U(t_n)-U(t)\|_{H_D^s(G)}\to 0$.
		
		Step 2. Extension to all $U\in X^{s,b}(G)$.
		Let $U\in X^{s,b}(G)$ and choose $U_k\in\mathcal{U}$ with $U_k\to U$ in $X^{s,b}$.
		By \eqref{eq2.7}, $(U_k)$ is Cauchy in $C(\mathbb{R};H_D^s(G))$, hence converges to some
		$V\in C(\mathbb{R};H_D^s(G))$. We redefine $U$ as this continuous representative $V$.
		Then \eqref{eq2.7} follows by passing to the limit.
		
		Step 3. Restriction to $[0,T]$.
		Let $u\in X_T^{s,b}$. By Definition~\ref{Def2.5}, for every $\varepsilon>0$ there exists an extension
		$U\in X^{s,b}(G)$ such that $U(t)=u(t)$ on $[0,T]$ and
		\[
		\|U\|_{X^{s,b}}\le \|u\|_{X_T^{s,b}}+\varepsilon.
		\]
		By Step 2, $U\in C(\mathbb{R};H_D^s(G))$, hence $u\in C([0,T];H_D^s(G))$.
		Moreover, using \eqref{eq2.7},
		\[
		\sup_{t\in[0,T]}\|u(t)\|_{H_D^s(G)}
		\le \sup_{t\in\mathbb{R}}\|U(t)\|_{H_D^s(G)}
		\le C_b\|U\|_{X^{s,b}}
		\le C_b\bigl(\|u\|_{X_T^{s,b}}+\varepsilon\bigr).
		\]
		Letting $\varepsilon\to 0$ gives \eqref{eq2.6}.
	\end{proof}
	
	\section{Linear and Nonlinear Estimates}\label{Sec3}
	
	Throughout this section, $X^{s,b}$ and $X_T^{s,b}$ are the Bourgain spaces associated with $D$
	as in Definitions~\ref{Def2.3}--\ref{Def2.5}.
	We fix a nonnegative cutoff $\eta\in C_c^\infty(\mathbb{R})$ such that
	\begin{equation}\label{eq3.1}
		\eta(t)=1\ \text{for }|t|\le 1,
		\qquad
		\mathrm{supp}\,\eta\subset\{t:|t|\le 2\}.
	\end{equation}
	For $T\in(0,1]$ we set $\eta_T(t)=\eta(t/T)$.
	
	\begin{lemma}\label{Lem3.1}
		Let $\eta\in C_c^\infty(\mathbb{R})$.
		For any $s\in\mathbb{R}$ and any $b\in\mathbb{R}$,
		\begin{equation}\label{eq3.2}
			\|\eta(t)e^{-itD}\psi_0\|_{X^{s,b}}\le C_{\eta,b}\|\psi_0\|_{H_D^s(G)}.
		\end{equation}
		Consequently, for the cutoff $\eta$ fixed in \eqref{eq3.1} and any $T\in(0,1]$,
		\begin{equation}\label{eq3.3}
			\|e^{-itD}\psi_0\|_{X_T^{s,b}}\le C_b\|\psi_0\|_{H_D^s(G)},
			\qquad C_b=C_{\eta,b}.
		\end{equation}
	\end{lemma}
	
	\begin{proof}
		Let $\psi_0\in H_D^s(G)$ and set $U(t)=\eta(t)e^{-itD}\psi_0$.
		Applying the spectral transform $\mathcal{F}_D$ in the space variable gives
		\[
		(\mathcal{F}_D U)(t,\lambda)=\eta(t)e^{-it\lambda}(\mathcal{F}_D\psi_0)(\lambda)
		\quad\text{in }\mathcal{H}_\lambda.
		\]
		Taking the Fourier transform in time yields
		\[
		\widetilde U(\tau,\lambda)
		=(\mathcal{F}_t\mathcal{F}_D U)(\tau,\lambda)
		=\widehat{\eta}(\tau+\lambda)\,(\mathcal{F}_D\psi_0)(\lambda),
		\]
		where $\widehat{\eta}=\mathcal{F}_t\eta$.
		Hence, by the definition of the $X^{s,b}$ norm,
		\begin{align*}
			\|U\|_{X^{s,b}}^2
			&=\int_{\mathbb{R}}\int_{\mathbb{R}}
			\langle\lambda\rangle^{2s}\langle\tau+\lambda\rangle^{2b}
			|\widehat{\eta}(\tau+\lambda)|^2
			\|(\mathcal{F}_D\psi_0)(\lambda)\|_{\mathcal{H}_\lambda}^2
			\,d\tau\,d\mu_D(\lambda)\\
			&=\left(\int_{\mathbb{R}}\langle\sigma\rangle^{2b}|\widehat{\eta}(\sigma)|^2\,d\sigma\right)
			\int_{\mathbb{R}}\langle\lambda\rangle^{2s}\|(\mathcal{F}_D\psi_0)(\lambda)\|_{\mathcal{H}_\lambda}^2\,d\mu_D(\lambda),
		\end{align*}
		where we used the change of variables $\sigma=\tau+\lambda$.
		By \eqref{eq2.4}, the second integral equals $\|\psi_0\|_{H_D^s(G)}^2$.
		This proves \eqref{eq3.2} with
		\[
		C_{\eta,b}^2=\int_{\mathbb{R}}\langle\sigma\rangle^{2b}|\widehat{\eta}(\sigma)|^2\,d\sigma.
		\]
		
		To prove \eqref{eq3.3}, assume that $\eta$ satisfies \eqref{eq3.1}.
		For $T\in(0,1]$, the function $U(t)=\eta(t)e^{-itD}\psi_0$ coincides with $e^{-itD}\psi_0$ on $[0,T]$,
		hence Definition~\ref{Def2.5} gives
		\[
		\|e^{-itD}\psi_0\|_{X_T^{s,b}}\le \|U\|_{X^{s,b}}
		\le C_{\eta,b}\|\psi_0\|_{H_D^s(G)}.
		\]
		Setting $C_b=C_{\eta,b}$ yields \eqref{eq3.3}.
	\end{proof}
	
	For $\lambda\in\mathbb{R}$ and $\theta\in\mathbb{R}$, we define the shifted Sobolev space
	$H_\lambda^\theta(\mathbb{R};\mathcal H)$ by
	\[
	\|g\|_{H_\lambda^\theta(\mathbb{R};\mathcal H)}^2
	=\int_{\mathbb{R}}\langle\tau+\lambda\rangle^{2\theta}\|\widehat g(\tau)\|_{\mathcal H}^2\,d\tau,
	\qquad
	\widehat g=\mathcal{F}_t g.
	\]
	
	\begin{lemma}\label{Lem3.2}
		Let $b\in\left(\frac12,1\right)$ and $b'\in\left(0,\frac12\right)$.
		Let $\eta$ satisfy \eqref{eq3.1} and let $T\in(0,1]$.
		For any $\lambda\in\mathbb{R}$ and any Hilbert space $\mathcal H$, define
		\[
		\mathcal T_\lambda f(t)
		=\eta_T(t)\int_0^t e^{-i(t-\tau)\lambda}f(\tau)\,d\tau,
		\qquad
		\eta_T(t)=\eta(t/T),
		\]
		for $f\in H_\lambda^{b'-1}(\mathbb{R};\mathcal H)$.
		Then there exists $C$ depending only on $b,b'$ and $\eta$ such that
		\begin{equation}\label{eq3.4}
			\|\mathcal T_\lambda f\|_{H_\lambda^{b}(\mathbb{R};\mathcal H)}
			\le C\,T^{\,1-b+b'}\|f\|_{H_\lambda^{b'-1}(\mathbb{R};\mathcal H)}.
		\end{equation}
		The constant $C$ is independent of $\lambda$ and $T$.
	\end{lemma}
	
	\begin{proof}
		Multiplication by $e^{it\lambda}$ reduces the estimate to the case $\lambda=0$. Indeed, with
		$g(t)=e^{it\lambda}\mathcal T_\lambda f(t)$ and $F(t)=e^{it\lambda}f(t)$ one has
		\[
		g(t)=\eta_T(t)\int_0^t F(\tau)\,d\tau,
		\qquad
		\|\mathcal T_\lambda f\|_{H_\lambda^b}=\|g\|_{H^b},
		\qquad
		\|f\|_{H_\lambda^{b'-1}}=\|F\|_{H^{b'-1}}.
		\]
		It remains to use the standard time-localized Bourgain Duhamel estimate
		\begin{equation}\label{eq3.5}
			\left\|\eta_T(t)\int_0^t F(\tau)\,d\tau\right\|_{H^b(\mathbb R;\mathcal H)}
			\le C T^{1-b+b'}\|F\|_{H^{b'-1}(\mathbb R;\mathcal H)},
		\end{equation}
		valid for $b\in(1/2,1)$, $b'\in(0,1/2)$ and $0<T\le1$; see, for example, Tao \cite[Proposition~2.12]{TaoBilinearDispersive}. For completeness we recall the scalar Fourier-side form of the argument. If $F$ is Schwartz, then
		\[
		\mathcal F_t\left(\eta_T(t)\int_0^t F(\tau)\,d\tau\right)(\tau)
		=\frac1{\sqrt{2\pi}}\int_{\mathbb R}
		\frac{\widehat\eta_T(\tau-\sigma)-\widehat\eta_T(\tau)}{i\sigma}\,\widehat F(\sigma)\,d\sigma,
		\]
		where the difference quotient is understood at $\sigma=0$ by continuity. The proof of
		\eqref{eq3.5} estimates this multiplier directly in $L^2$ after the usual low/high
		$\sigma$ decomposition. The high-frequency part is controlled by the cancellation in
		$\widehat\eta_T(\tau-\sigma)-\widehat\eta_T(\tau)$; one does not estimate the two terms separately by
		$\int_{|\sigma|>1}\langle\sigma\rangle^{-2b'}d\sigma$, which would diverge for $b'<1/2$. The Hilbert-valued version follows from the scalar estimate by expanding with respect to an orthonormal basis of $\mathcal H$ and using Plancherel. This gives \eqref{eq3.5}, and hence \eqref{eq3.4}.
	\end{proof}
	
	\begin{lemma}\label{Lem3.3}
		Let $b\in\left(\frac12,1\right)$ and let $b'\in\left(0,\frac12\right)$.
		For $T\in(0,1]$ and any $F\in X_T^{s,b'-1}$,
		\begin{equation}\label{eq3.6}
			\left\|\int_0^t e^{-i(t-\tau)D}F(\tau)\,d\tau\right\|_{X_T^{s,b}}
			\le C\,T^{\,1-b+b'}\|F\|_{X_T^{s,b'-1}},
		\end{equation}
		where $C$ depends only on $b,b'$ and $\eta$ in \eqref{eq3.1}.
	\end{lemma}
	
	\begin{proof}
		Let $F\in X_T^{s,b'-1}$ and fix $\varepsilon>0$.
		By Definition~\ref{Def2.5}, there exists an extension $\widetilde F\in X^{s,b'-1}$ such that
		$\widetilde F(t)=F(t)$ for a.e.\ $t\in[0,T]$ and
		\[
		\|\widetilde F\|_{X^{s,b'-1}}\le \|F\|_{X_T^{s,b'-1}}+\varepsilon.
		\]
		Set
		\[
		\Psi(t)=\eta_T(t)\int_0^t e^{-i(t-\tau)D}\widetilde F(\tau)\,d\tau.
		\]
		Since $T\in(0,1]$ and $\eta(t)=1$ for $|t|\le 1$, we have $\eta_T(t)=\eta(t/T)=1$ for all $t\in[0,T]$.
		Therefore, for a.e.\ $t\in[0,T]$,
		\[
		\Psi(t)=\int_0^t e^{-i(t-\tau)D}F(\tau)\,d\tau.
		\]
		Hence, by Definition~\ref{Def2.5},
		\[
		\left\|\int_0^t e^{-i(t-\tau)D}F(\tau)\,d\tau\right\|_{X_T^{s,b}}
		\le \|\Psi\|_{X^{s,b}}.
		\]
		
		Apply $\mathcal{F}_D$ in the space variable. For $\mu_D$-a.e.\ $\lambda$ set
		\[
		F_\lambda(t)=(\mathcal{F}_D\widetilde F)(t,\lambda)\in\mathcal{H}_\lambda,
		\qquad
		\Psi_\lambda(t)=(\mathcal{F}_D\Psi)(t,\lambda)
		=\eta_T(t)\int_0^t e^{-i(t-\tau)\lambda}F_\lambda(\tau)\,d\tau.
		\]
		By Lemma~\ref{Lem3.2} with $\mathcal H=\mathcal{H}_\lambda$,
		\[
		\|\Psi_\lambda\|_{H_\lambda^b(\mathbb{R};\mathcal{H}_\lambda)}
		\le C\,T^{\,1-b+b'}\|F_\lambda\|_{H_\lambda^{b'-1}(\mathbb{R};\mathcal{H}_\lambda)}.
		\]
		Multiply by $\langle\lambda\rangle^{s}$, square, and integrate in $\lambda$ with respect to $d\mu_D(\lambda)$.
		Using \eqref{eq2.5}, we obtain
		\[
		\|\Psi\|_{X^{s,b}}
		\le C\,T^{\,1-b+b'}\|\widetilde F\|_{X^{s,b'-1}}
		\le C\,T^{\,1-b+b'}\bigl(\|F\|_{X_T^{s,b'-1}}+\varepsilon\bigr).
		\]
		Letting $\varepsilon\to 0$ yields \eqref{eq3.6}.
	\end{proof}

	\begin{lemma}\label{Lem3.4}
		Let
		\[
		\mathbf 1_N=N^{-1/2}(1,\ldots,1)^{\mathsf T}\in\mathbb C^N,
		\]
		and choose $V\in U(N)$ such that $V\mathbf 1_N=e_1$. Define $\mathcal U=V\otimes I_{\mathbb C^2}$ on $L^2(G;\mathbb C^2)$. Then
		\[
		\mathcal U D\mathcal U^{-1}=D_+\oplus D_-^{\oplus(N-1)},
		\]
		where
		\[
		\operatorname{Dom}(D_\kappa)=
		\{z=(\varphi,\chi)^{\mathsf T}\in H^1(\mathbb R_+;\mathbb C^2):
		(I-\kappa\sigma_3)z(0)=0\},
		\qquad \kappa\in\{1,-1\}.
		\]
		Thus $D_+$ corresponds to $\chi(0)=0$, while $D_-$ corresponds to $\varphi(0)=0$.
		For a half-line channel set
		\[
		(\mathcal E_\kappa z)(x)=
		\begin{cases}
			2^{-1/2}z(x),&x\ge0,\\[3pt]
			2^{-1/2}\kappa\sigma_3z(-x),&x<0.
		\end{cases}
		\]
		Then $\mathcal E_\kappa$ is an $L^2$ isometry from $L^2(\mathbb R_+;\mathbb C^2)$ into $L^2(\mathbb R;\mathbb C^2)$ and intertwines $D_\kappa$ with the line Dirac operator $D_{\mathbb R}=-i\sigma_1\partial_x+m\sigma_3$ by functional calculus. Moreover, for $0\le s<1/2$,
		\begin{equation}\label{eq3.7}
			\|u\|_{H_D^s(G)}\simeq
			\left(\sum_{e=1}^N\|u_e\|_{H^s(\mathbb R_+;\mathbb C^2)}^2\right)^{1/2}.
		\end{equation}
		The constants depend only on $s,m$ and on the fixed graph.
	\end{lemma}
	
	\begin{proof}
		The Dirac--Kirchhoff conditions read
		\[
		\varphi_1(0)=\cdots=\varphi_N(0),
		\qquad \sum_{e=1}^N\chi_e(0)=0.
		\]
		After applying $V$ in the edge index, the first condition says that the transformed $\varphi$-components vanish in channels $j=2,\ldots,N$, while the second condition says that the transformed $\chi$-component vanishes in channel $j=1$. This gives
		\[
		\chi_1(0)=0,
		\qquad
		\varphi_j(0)=0,\quad j=2,\ldots,N,
		\]
		and hence the stated direct sum decomposition.
		
		The continuity of $\mathcal E_\kappa z$ at $x=0$ is equivalent to
		$z(0)=\kappa\sigma_3z(0)$, that is, to $(I-\kappa\sigma_3)z(0)=0$.  On $x<0$,
		\[
		D_{\mathbb R}(\kappa\sigma_3 z(-x))
		=\kappa\sigma_3(D_\kappa z)(-x),
		\]
		because $\sigma_1\sigma_3=-\sigma_3\sigma_1$ and the mass term is preserved by $\sigma_3^2=I$. This proves the intertwining on the operator domain, and then for Borel functions of the operators by the spectral theorem. Finally, since $\langle D_{\mathbb R}\rangle^s$ is equivalent to $\langle\partial_x\rangle^s$ on the line and $0\le s<1/2$ is below the trace threshold on the half-line, the reflection characterization gives \eqref{eq3.7}.
	\end{proof}
	
	\begin{lemma}\label{Lem3.5}
		For every $T>0$ there is a constant $C_T$ such that
		\[
		\|e^{-itD}f\|_{L^\infty([0,T]\times G)}
		\le C_T\|f\|_{L^\infty(G)}
		\]
		and
		\[
		\left\|\int_0^t e^{-i(t-\tau)D}F(\tau)\,d\tau\right\|_{L^\infty([0,T]\times G)}
		\le C_T\|F\|_{L^1(0,T;L^\infty(G))}.
		\]
	\end{lemma}
	
	\begin{proof}
		By Lemma~\ref{Lem3.4} it is enough to prove the estimate for the line Dirac flow, up to constants depending on the fixed finite-dimensional edge transformation.  In the eigenbasis of $\sigma_1$, the equation
		\[
		i\partial_t u=(-i\sigma_1\partial_x+m\sigma_3)u
		\]
		is a first-order hyperbolic system whose principal part is transport with speeds $\pm1$ and whose zeroth-order coefficient has norm bounded by $m$.  Gronwall's inequality along characteristics gives
		\[
		\|e^{-itD_{\mathbb R}}f\|_{L^\infty(\mathbb R)}\le e^{m|t|}\|f\|_{L^\infty(\mathbb R)}.
		\]
		The half-line and graph estimates follow from the reflection and the finite-dimensional transformation.  The Duhamel estimate follows by applying the free estimate to $e^{-i(t-\tau)D}F(\tau)$ and integrating in $\tau$.
	\end{proof}
	
	\begin{lemma}\label{Lem3.6}
		Let $p\ge3$, $0\le s<1/2$, and $g(z)=|z|^{p-2}z$ for $z\in\mathbb C^2$. Then
		\begin{equation}\label{eq3.8}
			\|g(u)\|_{H_D^s(G)}
			\le C\|u\|_{L^\infty(G)}^{p-2}\|u\|_{H_D^s(G)}
		\end{equation}
		for every $u\in H_D^s(G)\cap L^\infty(G)$. Moreover, for every $R>0$ there is $C_R>0$ such that, whenever
		\[
		\|u\|_{H_D^s}+\|u\|_{L^\infty}
		+\|v\|_{H_D^s}+\|v\|_{L^\infty}\le R,
		\]
		one has
		\begin{equation}\label{eq3.9}
			\|g(u)-g(v)\|_{H_D^s(G)}
			\le C_R\bigl(\|u-v\|_{H_D^s(G)}+\|u-v\|_{L^\infty(G)}\bigr).
		\end{equation}
	\end{lemma}
	
	\begin{proof}
		By \eqref{eq3.7}, it suffices to work on one half-line.  The case $s=0$ follows from the pointwise bounds
		\[
		|g(a)|\le C|a|^{p-2}|a|,
		\qquad
		|g(a)-g(b)|\le C_R|a-b|,
		\qquad |a|,|b|\le R.
		\]
		Let $0<s<1/2$.  We use the Slobodeckij seminorm on $H^s(\mathbb R_+)$.  Since $p\ge3$, the map $g$ is $C^1$ and $Dg$ is Lipschitz on bounded subsets of $\mathbb C^2$. Thus
		\[
		|g(a)-g(b)|\le C M^{p-2}|a-b|\quad\text{if } |a|,|b|\le M,
		\]
		and
		\[
		|g(a)-g(b)|\le C_R|a-b|,
		\qquad
		\|Dg(a)-Dg(b)\|\le C_R|a-b|,
		\qquad |a|,|b|\le R.
		\]
		Taking $M=\|u\|_{L^\infty}$ in the first inequality gives \eqref{eq3.8}. For the difference estimate, set $h=u-v$ and write
		\[
		g(u)-g(v)=\int_0^1 Dg(v+\theta h)h\,d\theta.
		\]
		For $x,y\in\mathbb R_+$, the preceding Lipschitz bounds give
		\[
		\begin{aligned}
			&|[g(u)-g(v)](x)-[g(u)-g(v)](y)|\\
			&\quad \le C_R|h(x)-h(y)|
			+C_R|h(y)|\bigl(|u(x)-u(y)|+|v(x)-v(y)|\bigr).
		\end{aligned}
		\]
		After dividing by $|x-y|^{1/2+s}$ and integrating in $(x,y)$, the second term is bounded by
		\[
		C_R\|h\|_{L^\infty}\bigl(\|u\|_{H^s(\mathbb R_+)}+\|v\|_{H^s(\mathbb R_+)}\bigr).
		\]
		Together with the $L^2$ bound and the edgewise summation, this proves \eqref{eq3.9}.
	\end{proof}
	
	For $T>0$ and $b>1/2$ we set
	\[
	Y_T^{s,b}=X_T^{s,b}(G)\cap L^\infty([0,T]\times G),
	\qquad
	\|u\|_{Y_T^{s,b}}=\|u\|_{X_T^{s,b}}+
	\|u\|_{L^\infty([0,T]\times G)}.
	\]
	
	\begin{lemma}\label{Lem3.7}
		Let $p\ge3$, $0\le s<1/2$, $b>1/2$ and $0<b'<1/2$. Then, for every $T\in(0,1]$ and every $u\in Y_T^{s,b}$,
		\[
		\||u|^{p-2}u\|_{X_T^{s,b'-1}}
		\le C T^{1/2}\|u\|_{L^\infty([0,T]\times G)}^{p-2}\|u\|_{X_T^{s,b}}.
		\]
		Moreover, for every $R>0$ there is $C_R>0$ such that, if $u,v\in Y_T^{s,b}$ and
		$\|u\|_{Y_T^{s,b}}+\|v\|_{Y_T^{s,b}}\le R$, then
		\[
		\||u|^{p-2}u-|v|^{p-2}v\|_{X_T^{s,b'-1}}
		\le C_R T^{1/2}\|u-v\|_{Y_T^{s,b}}.
		\]
	\end{lemma}
	
	\begin{proof}
		Let $g(z)=|z|^{p-2}z$. Since $b'-1<0$, the zero extension from $[0,T]$ to $\mathbb R$ gives
		\[
		\|F\|_{X_T^{s,b'-1}}
		\le \|F\|_{L^2(0,T;H_D^s(G))}.
		\]
		Using Lemma~\ref{Lem3.6} at each time and the embedding
		$X_T^{s,b}\hookrightarrow C([0,T];H_D^s(G))$, we obtain
		\[
		\|g(u)\|_{X_T^{s,b'-1}}
		\le \|g(u)\|_{L^2(0,T;H_D^s)}
		\le C T^{1/2}\|u\|_{L^\infty([0,T]\times G)}^{p-2}
		\|u\|_{L^\infty(0,T;H_D^s)}
		\le C T^{1/2}\|u\|_{L^\infty}^{p-2}\|u\|_{X_T^{s,b}}.
		\]
		The difference estimate follows in the same way from \eqref{eq3.9}.
	\end{proof}

	\section{Proof of Local Well-Posedness}\label{Sec4}
	
	\begin{lemma}\label{Lem4.1}
		Let $p\ge3$, $0\le s<\frac12$, $b>1/2$ and $0<b'<1/2$.
		Assume that $\psi\in Y_T^{s,b}$ satisfies, for every $t\in[0,T]$,
		\[
		\psi(t)=e^{-itD}\psi_0+i\int_0^t e^{-i(t-\tau)D}\,|\psi(\tau)|^{p-2}\psi(\tau)\,d\tau
		\quad\text{in }H_D^s(G).
		\]
		Then $\psi$ is a distributional solution of \eqref{eq1.1} on $(0,T)$ with values in $H_D^{s-1}(G)$, namely
		\[
		i\partial_t\psi=D\psi-|\psi|^{p-2}\psi
		\quad\text{in }\mathcal D'\bigl((0,T);H_D^{s-1}(G)\bigr).
		\]
	\end{lemma}
	
	\begin{proof}
		Since $b>\frac12$, Lemma~\ref{Lem2.7} yields $\psi\in C([0,T];H_D^s(G))$.
		By Lemma~\ref{Lem3.7},
		\[
		F(t)=|\psi(t)|^{p-2}\psi(t)\in X_T^{s,b'-1}.
		\]
		
		Let $\varepsilon>0$.
		Choose an extension $\widetilde F\in X^{s,b'-1}$ of $F$ from $[0,T]$ such that
		\[
		\widetilde F(t)=F(t)\ \text{for }t\in[0,T],
		\qquad
		\|\widetilde F\|_{X^{s,b'-1}}\le \|F\|_{X_T^{s,b'-1}}+\varepsilon.
		\]
		Define
		\[
		\Psi(t)=\eta_T(t)\int_0^t e^{-i(t-\tau)D}\widetilde F(\tau)\,d\tau,
		\qquad
		\eta_T(t)=\eta(t/T),
		\]
		with $\eta$ as in \eqref{eq3.1}.
		Since $\eta_T(t)=1$ for $t\in[0,T]$, we have
		\[
		\Psi(t)=\int_0^t e^{-i(t-\tau)D}F(\tau)\,d\tau
		\qquad\text{for all }t\in[0,T].
		\]
		Arguing as in the proof of Lemma~\ref{Lem3.3}, we obtain
		\[
		\Psi\in X^{s,b},
		\qquad
		\|\Psi\|_{X^{s,b}}\le C\,T^{\,1-b+b'}\|\widetilde F\|_{X^{s,b'-1}}.
		\]
		In particular, restricting to $[0,T]$ and using Lemma~\ref{Lem2.7}, we have $\Psi\in C([0,T];H_D^s(G))$.
		The mild identity can be rewritten as
		\[
		\psi(t)=e^{-itD}\psi_0+i\Psi(t)\qquad\text{in }H_D^s(G)\ \text{for }t\in[0,T].
		\]
		
		Let $\zeta\in C_c^\infty(0,T)$ and $\phi\in\mathrm{Dom}(D)$.
		Since $t\mapsto e^{-itD}\psi_0$ satisfies $i\partial_t(e^{-itD}\psi_0)=D(e^{-itD}\psi_0)$ in $H_D^{s-1}(G)$,
		integration by parts on $(0,T)$ gives
		\begin{equation}\label{eq4.1}
			\int_0^T \bigl\langle e^{-itD}\psi_0,\,-i\zeta'(t)\phi-\zeta(t)D\phi\bigr\rangle_{L^2(G)}\,dt=0.
		\end{equation}
		
		It remains to prove the corresponding identity for $\Psi$ and $F$.
		Apply $\mathcal F_D$ in the space variable.
		For $\mu_D$-a.e.\ $\lambda$, set
		\[
		F_\lambda(t)=(\mathcal F_D\widetilde F)(t,\lambda)\in\mathcal H_\lambda,
		\qquad
		\Psi_\lambda(t)=(\mathcal F_D\Psi)(t,\lambda)\in\mathcal H_\lambda,
		\qquad
		\phi_\lambda=(\mathcal F_D\phi)(\lambda)\in\mathcal H_\lambda.
		\]
		Then
		\[
		\Psi_\lambda(t)=\eta_T(t)\int_0^t e^{-i(t-\tau)\lambda}F_\lambda(\tau)\,d\tau.
		\]
		
		Fix such a $\lambda$.
		Choose a sequence $F_{\lambda,n}\in C_c^\infty(\mathbb{R};\mathcal H_\lambda)$ such that
		$F_{\lambda,n}\to F_\lambda$ in $H_\lambda^{b'-1}(\mathbb{R};\mathcal H_\lambda)$, and let
		\[
		\Psi_{\lambda,n}(t)=\eta_T(t)\int_0^t e^{-i(t-\tau)\lambda}F_{\lambda,n}(\tau)\,d\tau.
		\]
		Since $\eta_T\equiv 1$ on $[0,T]$, for $t\in(0,T)$ we have
		\[
		\Psi_{\lambda,n}(t)=\int_0^t e^{-i(t-\tau)\lambda}F_{\lambda,n}(\tau)\,d\tau,
		\]
		hence $\Psi_{\lambda,n}\in C^1((0,T);\mathcal H_\lambda)$ and satisfies on $(0,T)$
		\[
		i\partial_t\Psi_{\lambda,n}=\lambda\Psi_{\lambda,n}+F_{\lambda,n}.
		\]
		Multiplying by $\zeta(t)$, pairing with $\phi_\lambda$, and integrating by parts on $(0,T)$ gives
		\begin{equation}\label{eq4.2}
			\int_0^T \bigl\langle \Psi_{\lambda,n}(t),\,-i\zeta'(t)\phi_\lambda-\zeta(t)\lambda\phi_\lambda\bigr\rangle_{\mathcal H_\lambda}\,dt
			+\int_0^T \langle F_{\lambda,n}(t),\,\zeta(t)\phi_\lambda\rangle_{\mathcal H_\lambda}\,dt=0.
		\end{equation}
		
		By Lemma~\ref{Lem3.2}, the map
		$f\mapsto \eta_T(t)\int_0^t e^{-i(t-\tau)\lambda}f(\tau)\,d\tau$
		is continuous from $H_\lambda^{b'-1}$ to $H_\lambda^{b}$, hence $\Psi_{\lambda,n}\to \Psi_\lambda$ in
		$H_\lambda^b(\mathbb{R};\mathcal H_\lambda)$.
		Since $b>1/2$, the convergence $\Psi_{\lambda,n}\to\Psi_\lambda$ in $H_\lambda^b$ implies convergence in $L^2(0,T;\mathcal H_\lambda)$, and hence the first term in \eqref{eq4.2} passes to the limit.
		For the second term, $F_{\lambda,n}\to F_\lambda$ in $H_\lambda^{b'-1}$, while $t\mapsto \zeta(t)\phi_\lambda$ belongs to $H_\lambda^{1-b'}(\mathbb R;\mathcal H_\lambda)$.  Thus the pairing converges by the duality
		$(H_\lambda^{b'-1})^*=H_\lambda^{1-b'}$.  Letting $n\to\infty$ yields
		\begin{equation}\label{eq4.3}
			\int_0^T \bigl\langle \Psi_{\lambda}(t),\,-i\zeta'(t)\phi_\lambda-\zeta(t)\lambda\phi_\lambda\bigr\rangle_{\mathcal H_\lambda}\,dt
			+\int_0^T \langle F_{\lambda}(t),\,\zeta(t)\phi_\lambda\rangle_{\mathcal H_\lambda}\,dt=0.
		\end{equation}
		
		Integrating \eqref{eq4.3} in $\lambda$ with respect to $\mu_D$ and using Plancherel for $\mathcal F_D$
		together with $(\mathcal F_D(D\phi))(\lambda)=\lambda\phi_\lambda$, we obtain
		\begin{equation}\label{eq4.4}
			\int_0^T \bigl\langle \Psi(t),\,-i\zeta'(t)\phi-\zeta(t)D\phi\bigr\rangle_{L^2(G)}\,dt
			+\int_0^T \langle \widetilde F(t),\,\zeta(t)\phi\rangle_{L^2(G)}\,dt=0.
		\end{equation}
		Since $\zeta$ is supported in $(0,T)$ and $\widetilde F=F$ on $[0,T]$, the last term equals
		$\int_0^T\langle F(t),\zeta(t)\phi\rangle_{L^2(G)}\,dt$.
		
		Adding \eqref{eq4.1} and \eqref{eq4.4} and using $\psi=e^{-itD}\psi_0+i\Psi$ on $[0,T]$ yields
		\[
		\int_0^T \bigl\langle \psi(t),-i\zeta'(t)\phi-\zeta(t)D\phi\bigr\rangle_{L^2(G)}\,dt
		+\int_0^T \langle |\psi(t)|^{p-2}\psi(t),\zeta(t)\phi\rangle_{L^2(G)}\,dt=0.
		\]
		
		Finally, since $\mathrm{Dom}(D)$ is dense in $H_D^{1-s}(G)$,
		the above identity extends by density to all $\phi\in H_D^{1-s}(G)$, which is the dual space of $H_D^{s-1}(G)$
		with respect to the $L^2$ pairing. Hence
		\[
		i\partial_t\psi=D\psi-|\psi|^{p-2}\psi
		\quad\text{in }\mathcal D'\bigl((0,T);H_D^{s-1}(G)\bigr).
		\]
	\end{proof}

	\begin{proof}[Proof of Theorem~\ref{Thm1.1}]
		Let $p\ge3$ and $0\le s<1/2$. Choose, for instance, fixed parameters
		\[
		b\in(1/2,1),\qquad b'\in(0,1/2).
		\]
		For $T\in(0,1]$ define, for $\psi\in Y_T^{s,b}$,
		\[
		\mathcal T(\psi)(t)=e^{-itD}\psi_0+i\int_0^t e^{-i(t-\tau)D}\,|\psi(\tau)|^{p-2}\psi(\tau)\,d\tau.
		\]
		Set
		\[
		A=\|\psi_0\|_{H_D^s(G)}+\|\psi_0\|_{L^\infty(G)},
		\qquad R=2C_0A,
		\]
		where $C_0$ is chosen larger than the constants in the free $X^{s,b}$ estimate, the free $L^\infty$ estimate on $[0,1]$, and the embedding $X_T^{s,b}\hookrightarrow C([0,T];H_D^s)$.
		Let
		\[
		B_R=\{\psi\in Y_T^{s,b}:\ \|\psi\|_{Y_T^{s,b}}\le R\}.
		\]
		The space $Y_T^{s,b}$ is Banach, hence $B_R$ is complete.
		
		For $\psi\in B_R$, Lemmas~\ref{Lem3.1}, \ref{Lem3.3} and \ref{Lem3.7} give
		\[
		\begin{aligned}
			\|\mathcal T(\psi)\|_{X_T^{s,b}}
			&\le C\|\psi_0\|_{H_D^s}
			+C T^{1-b+b'}\||\psi|^{p-2}\psi\|_{X_T^{s,b'-1}}\\
			&\le C\|\psi_0\|_{H_D^s}+C_R T^{1-b+b'+1/2}R.
		\end{aligned}
		\]
		Similarly, using Lemma~\ref{Lem3.5},
		\[
		\begin{aligned}
			\|\mathcal T(\psi)\|_{L^\infty([0,T]\times G)}
			&\le C\|\psi_0\|_{L^\infty}+C\||\psi|^{p-2}\psi\|_{L^1(0,T;L^\infty)}\\
			&\le C\|\psi_0\|_{L^\infty}+C T R^{p-1}.
		\end{aligned}
		\]
		Choosing $T>0$ sufficiently small, depending only on $A$ and the fixed graph, we get
		$\mathcal T(B_R)\subset B_R$.
		
		For $\psi_1,\psi_2\in B_R$, Lemmas~\ref{Lem3.3} and \ref{Lem3.7} yield
		\[
		\|\mathcal T(\psi_1)-\mathcal T(\psi_2)\|_{X_T^{s,b}}
		\le C_R T^{1-b+b'+1/2}\|\psi_1-\psi_2\|_{Y_T^{s,b}}.
		\]
		The pointwise Lipschitz bound for $g(z)=|z|^{p-2}z$ on bounded sets and Lemma~\ref{Lem3.5} give
		\[
		\|\mathcal T(\psi_1)-\mathcal T(\psi_2)\|_{L^\infty([0,T]\times G)}
		\le C_R T\|\psi_1-\psi_2\|_{Y_T^{s,b}}.
		\]
		After decreasing $T$ if necessary, $\mathcal T$ is a contraction on $B_R$. Therefore it has a unique fixed point
		\[
		\psi\in Y_T^{s,b}=X_T^{s,b}\cap L^\infty([0,T]\times G).
		\]
		By Lemma~\ref{Lem2.7}, $\psi\in C([0,T];H_D^s(G))$.
		The fixed point identity gives the Duhamel formula in $H_D^s(G)$ for every $t\in[0,T]$, and Lemma~\ref{Lem4.1} implies that $\psi$ is a distributional solution of \eqref{eq1.1} on $(0,T)$ with values in $H_D^{s-1}(G)$.
		
		The same estimates, with $R$ chosen from a common bound for two initial data, give uniqueness and local Lipschitz dependence in the $Y_T^{s,b}$ norm. In particular,
		\[
		\sup_{t\in[0,T]}\|\psi(t)-\widetilde\psi(t)\|_{H_D^s(G)}
		+\|\psi-\widetilde\psi\|_{L^\infty([0,T]\times G)}
		\le C_R\bigl(\|\psi_0-\widetilde\psi_0\|_{H_D^s(G)}+\|\psi_0-\widetilde\psi_0\|_{L^\infty(G)}\bigr).
		\]
		This proves the local well-posedness part of the theorem.
	\end{proof}

	\section{Mass Conservation and Blow-up Alternative}\label{Sec5}
	In this section we derive two dynamical consequences of Theorem~\ref{Thm1.1}: conservation of the $L^2$ norm and a blow-up alternative. The Hamiltonian energy considered in \cite{Borrelli2021} belongs to the operator-domain theory; the conservation law established below is the charge conservation identity.

	\begin{lemma}\label{Lem5.1}
		Let $\psi$ be the solution given by Theorem~\ref{Thm1.1} with initial data
		$\psi_0\in H_D^s(G)\cap L^\infty(G)$. Then
		\[
		\|\psi(t)\|_{L^2(G;\mathbb{C}^2)}=\|\psi_0\|_{L^2(G;\mathbb{C}^2)}
		\quad\text{for all }t\in[0,T].
		\]
	\end{lemma}
	
	\begin{proof}
		Since $s\ge0$ and $b>1/2$, Lemma~\ref{Lem2.7} gives
		$\psi\in C([0,T];L^2(G;\mathbb C^2))$.  Moreover, Theorem~\ref{Thm1.1} gives
		$\psi\in L^\infty([0,T]\times G)$.  Therefore
		\[
		F:=|\psi|^{p-2}\psi\in L^\infty(0,T;L^2(G;\mathbb C^2))
		\subset L^1([0,T];L^2(G;\mathbb C^2)).
		\]
		
		For $M\ge1$ set $P_M=E_D([-M,M])$ and $\psi_M=P_M\psi$.  Then $P_M$ is an orthogonal projection on $L^2(G;\mathbb C^2)$, it commutes with $D$, and $D$ is bounded on $\mathrm{Ran}(P_M)$ with
		$\|DP_M\|_{L^2\to L^2}\le M$.  Applying $P_M$ to the Duhamel formula gives, for every $t\in[0,T]$,
		\[
		\psi_M(t)=e^{-itD}P_M\psi_0+i\int_0^t e^{-i(t-\tau)D}P_MF(\tau)\,d\tau
		\quad\text{in }L^2(G;\mathbb C^2).
		\]
		Since $DP_M$ is bounded and $P_MF\in L^1([0,T];L^2)$, the map $t\mapsto\psi_M(t)$ is absolutely continuous in $L^2$ and satisfies, for a.e. $t\in(0,T)$,
		\[
		i\partial_t\psi_M=D\psi_M-P_MF
		\quad\text{in }L^2(G;\mathbb C^2).
		\]
		For such $t$,
		\[
		\frac{d}{dt}\|\psi_M(t)\|_{L^2}^2
		=2\mathrm{Re}\,\langle -iD\psi_M(t)+iP_MF(t),\psi_M(t)\rangle_{L^2}.
		\]
		The contribution of $D$ vanishes because $D$ is self-adjoint and $\psi_M(t)\in\mathrm{Ran}(P_M)\subset\mathrm{Dom}(D)$.  Since $P_M$ is self-adjoint and $P_M\psi_M=\psi_M$,
		\[
		\langle P_MF,\psi_M\rangle_{L^2}=\langle F,\psi_M\rangle_{L^2}.
		\]
		Using $\psi_M=\psi-(I-P_M)\psi$ and
		\[
		\langle F(t),\psi(t)\rangle_{L^2}=
		\int_G |\psi(t,x)|^p\,dx\in\mathbb R,
		\]
		we obtain, for a.e. $t\in(0,T)$,
		\[
		\left|\frac{d}{dt}\|\psi_M(t)\|_{L^2}^2\right|
		\le 2\|F(t)\|_{L^2}\,\|(I-P_M)\psi(t)\|_{L^2}.
		\]
		Integrating in time gives
		\[
		\left|\|\psi_M(t)\|_{L^2}^2-\|P_M\psi_0\|_{L^2}^2\right|
		\le 2\int_0^t \|F(\tau)\|_{L^2}\,\|(I-P_M)\psi(\tau)\|_{L^2}\,d\tau.
		\]
		The set $\psi([0,T])$ is compact in $L^2(G;\mathbb C^2)$.  Since $P_M\to I$ strongly on $L^2$ and
		$\sup_M\|I-P_M\|_{L^2\to L^2}\le1$, the convergence is uniform on this compact set:
		\[
		\sup_{\tau\in[0,T]}\|(I-P_M)\psi(\tau)\|_{L^2}\to0.
		\]
		Together with $F\in L^1([0,T];L^2)$, this implies that the right-hand side tends to zero as $M\to\infty$.
		Also $\psi_M(t)=P_M\psi(t)\to\psi(t)$ and $P_M\psi_0\to\psi_0$ in $L^2$.  Passing to the limit yields
		\[
		\|\psi(t)\|_{L^2(G;\mathbb C^2)}^2=\|\psi_0\|_{L^2(G;\mathbb C^2)}^2,
		\qquad t\in[0,T].
		\]
		This proves the conservation law.
	\end{proof}
	
	\begin{lemma}\label{Lem5.2}
		Let $[0,T^\ast)$ be the maximal forward lifespan of the solution constructed in Theorem~\ref{Thm1.1}. If $T^\ast<\infty$ and
		\[
		\sup_{t<T^\ast}\|\psi(t)\|_{H_D^s(G)}+\|\psi\|_{L^\infty([0,T^\ast)\times G)}<\infty,
		\]
		then the solution extends beyond $T^\ast$.
	\end{lemma}
	
	\begin{proof}
		Assume that the displayed quantity is finite and denote it by $M$.  The local existence time in the proof of Theorem~\ref{Thm1.1} depends only on the size of the initial datum in $H_D^s\cap L^\infty$ and on the fixed graph. Hence there exists $\delta_0>0$, depending only on $M$ and on the graph, such that the contraction argument gives a solution on a time interval of length $\delta_0$ for every initial datum $\phi$ satisfying
		\[
		\|\phi\|_{H_D^s(G)}+\|\phi\|_{L^\infty(G)}\le M+1.
		\]
		Choose $t_0<T^\ast$ so close to $T^\ast$ that $T^\ast-t_0<\delta_0$.  Since $\psi\in L^\infty([0,T^\ast)\times G)$, we may choose such a $t_0$ which is also a Lebesgue time for the $L^\infty$ representative; then
		\[
		\|\psi(t_0)\|_{H_D^s(G)}+\|\psi(t_0)\|_{L^\infty(G)}\le M+1.
		\]
		Applying the local theory with initial datum $\psi(t_0)$ gives a solution on $[t_0,t_0+\delta_0]$. By uniqueness in the class $Y^{s,b}$, this solution coincides with the original one on the overlap. Since $t_0+\delta_0>T^\ast$, it extends $\psi$ beyond $T^\ast$, contradicting maximality. Therefore the stated boundedness condition cannot hold when $T^\ast<\infty$.
	\end{proof}

	\section*{Acknowledgments}
We would like to thank the anonymous referee for his/her careful readings of our manuscript and the useful comments. 
	
	\medskip
	{\bf Funding:} This work was supported by the National Natural Science Foundation of China
	(12301145, 12561020, 12261107) and Yunnan Fundamental Research Projects
	(202401AU070123, 202601AT070048).
	
	\medskip
	{\bf Author Contributions:} All the authors wrote the main manuscript text together and these authors contributed equally to this work.
	
	\medskip
	\textbf{Data availability.}
	Data sharing is not applicable to this article, since no new data were created or analyzed in this study.
	
	\medskip
	{\bf Conflict of Interests:} The authors declare that there is no conflict of interest.

\end{document}